\documentclass[12pt]{article}

\usepackage{amssymb}
\usepackage{amsthm}
\usepackage{amsmath}
\usepackage{graphicx}
\usepackage{fullpage}
\usepackage{color}
 \numberwithin{equation}{section}

\theoremstyle{plain}
\newtheorem{thm}{Theorem}[section]

\theoremstyle{definition}

\theoremstyle{remark}

\newcommand{\N}{\mathbb{N}}
\newcommand{\R}{\mathbb{R}}

\newcommand{\I}{\infty}
\newcommand{\bp}{\begin{proof}[\ensuremath{\mathbf{Proof}}]}
\newcommand{\bs}{\begin{proof}[\ensuremath{\mathbf{Solution}}]}
\newcommand{\ep}{\end{proof}}

\begin{document}

\title{A blowup criterion for ideal viscoelastic flow}

\author{Xianpeng Hu and Ryan Hynd\thanks{This material is based upon work supported by the National Science Foundation under Grant No. DMS-1004733.}\\
Courant Institute of Mathematical Sciences\\
New York University\\
251 Mercer Street\\
New York, NY 10012-1185 USA}
\maketitle

\setcounter{section}{1}
\begin{abstract}
We establish an analog of the Beale-Kato-Majda  criterion for singularities of smooth solutions of the system of PDE arising in the Oldroyd model for ideal viscoelastic flow.
\end{abstract}

\par It is well known that smooth solutions of the initial value problem associated with Euler's equations
\begin{equation}
\begin{cases}
\partial_tu +(u\cdot \nabla)u = -\nabla p\\
\hspace{.65in}\nabla \cdot u= 0
\end{cases},\quad (x,t)\in \R^3\times (0,T)
\end{equation}
exist for some finite time $T>0.$  Here $u=u(x,t)\in \R^3$ and $p=p(x,t)\in\R$ represent the local velocity and pressure of a perfect fluid, respectively. One of the major challenges in PDE theory is to deduce whether or not
finite time singularities do indeed occur.\footnote{A singularity is a time $T>0$ where $u(T)$ fails belong to a natural function space such as $C^{\infty}(\R^3;\R^3)$.}  A celebrated result of Beale, Kato,
and Majda asserts that if a solution $u$ of Euler's equations possesses a singularity at a finite time $T$, then necessarily
$$
\int^T_0 |(\nabla \times u)(t)|_{L^\I(\R^3)}dt=+\infty
$$
\cite{BKM}. In particular, the $L^\I(\R^3)$ norm of the vorticity must blow up as $t$ approaches $T$ from below
$$
\limsup_{t\rightarrow T^-}|(\nabla\times u)(t)|_{L^\I(\R^3)}=+\infty.
$$

\par Along with relatively standard energy estimates for $u$ and $\nabla\times u$, the key insight for Beale, Kato and Majda was a clever application of Kato's inequality:
\begin{equation}\label{BIOT}
|\nabla f|_{L^\I(\R^3)}\le C\left\{1 + (1+\ln^+|f|_{H^3(\R^3)})|\nabla \times f|_{L^\I(\R^3)} + |\nabla \times f|_{L^2(\R^3)}\right\}
\end{equation}
for all $f\in H^3(\R^3; \R^3)$ with
$$
\nabla\cdot f=0
$$
\cite{KATO}. Here $C$ is a universal constant,  and
$$
\ln^+ x=
\begin{cases}
\ln x, \quad x\ge 1\\
0,\quad x\le 1
\end{cases}.
$$
The above inequality was derived from a careful analysis of the Biot-Savart law
$$
f(x)= \int_{\R^3}\nabla\phi(x-y)\times (\nabla\times f(y))dy, \quad x\in \R^3,
$$
where $\phi(x)=1/4\pi |x|$ is the fundamental solution of Laplace's equation on $\R^3.$

\par Caflisch, Klapper, and Steel used these ideas to prove a similar blow up criterion for the ideal MHD equations
\begin{equation}
\begin{cases}
\partial_tu +(u\cdot \nabla)u = -\nabla\left(p +|b|^2/2\right) + (b\cdot \nabla)b\\
\; \partial_tb +(u\cdot \nabla)b = (b\cdot \nabla)u \\
\hspace{.65in}\nabla \cdot u= \nabla\cdot b=0
\end{cases},\quad (x,t)\in\R^3\times (0,T).
\end{equation}
They showed that if the solution to the inital value problem, with smooth divergence free data, has a singularity at time $T$, then it must be that
$$
\int^T_0 |(\nabla \times u)(t)|_{L^\I(\R^3)}+ |(\nabla \times b)(t)|_{L^\I(\R^3)}dt=+\infty
$$
\cite{Fish}.

\par In this short note, we develop these ideas further and establish an analogous blowup criterion for solutions of the following system of PDE

\begin{equation}\label{IDEAL}
\begin{cases}
\;\; \partial_t u +(u\cdot \nabla)u = -\nabla p + \nabla\cdot FF^t\\
\partial_t F +(u\cdot \nabla)F =  \nabla u F \\
\hspace{.73in}\nabla \cdot u= 0
\end{cases}, \quad (x,t)\in\R^3\times (0,T).
\end{equation}
This system of PDE arises in the Oldroyd
model for ideal viscoelastic flow i.e. a viscoelastic fluid whose
elastic properties dominate its behavior.  Here $F=F(x,t)\in\R^{3\times 3}$ represents the local deformation
gradient of the fluid.  See \cite{LLZ} for more
on this model.

\par We remark that global existence for solutions  near
equilibrium of the viscous analog of \eqref{IDEAL} has been verified by Lin et.al. \cite{LLZ} for two-dimensional flow and Lei et.al. \cite{Lei} for three-dimensional flow; 
see also the work Lin and Zhang \cite{Zhang} for a further discussion of these topics. In these references, three
assumed properties of the deformation gradient are
crucial for the methods employed: 
\begin{itemize}
\item The condition $\nabla\cdot F^t=0$ (see equation \eqref{divFtranspose} below);
\item The curl of the deformation gradient is of ``higher order" (see equation (3) in \cite{Lei});
\item The condition $\textrm{det} F=1$.
\end{itemize}
Note that the third condition is natural as the flow is assumed to be
incompressible. For the vanishing viscosity case,
Sideris and Thomases \cite{Sideris} also established the global existence
of smooth solutions near the equilibrium using the incompressible
limit by imposing a null condition on the elastic stress.

\par One motivation of this work is to consider which conditions in the
above list are necessary to assume when considering global smooth solutions 
without any type of smallness requirement.  The central result of this work suggests that 
even if the incompressibility of the deformation gradient is
imposed, smooth solutions of \eqref{IDEAL} may become singular in finite time. From this standpoint, the property that the curl of the deformation gradient is higher order is perhaps the most important when considering
global existence. 
\\
\par Before stating our main result, let us make a simple observation alluded to above.  Letting $F_k=Fe_k$ denote the columns of $F$ we take the divergence the second equation in \eqref{IDEAL} to arrive at
\begin{equation}\label{divFtranspose}
\partial_t(\nabla \cdot F_k) + (u\cdot \nabla)(\nabla \cdot F_k)=0, \quad k=1,2,3.
\end{equation}
Therefore, if $\nabla \cdot F_k=0$ initially, it will remain so for later times.  In what follows, we will make this assumption and also note that it implies the equality
$$
\nabla\cdot FF^t=\sum^3_{k=1}(F_k\cdot \nabla) F_k.
$$
With this reduction, equation \eqref{IDEAL} becomes
\begin{equation}\label{newIDEAL}
\begin{cases}
\hspace{.175in}\partial_tu +(u\cdot \nabla)u = -\nabla p + \sum^3_{k=1}(F_k\cdot \nabla) F_k\\
\partial_tF_k +(u\cdot \nabla)F_k =  (F_k\cdot\nabla )u \\
\hspace{.82in}\nabla \cdot u= \nabla \cdot F_k=0
\end{cases}
\end{equation}
for  $k=1,2,3$.   We are now ready to state and prove our main result.

\begin{thm}\label{mainthm}
Assume that
$$
u_0\in H^s(\R^3)\quad \text{and}\quad F_{k,0}\in H^s(\R^3)
$$
for $k=1,2,3$ and some $s\in \N, s\ge 3$. Furthermore, assume
$$
\nabla\cdot u_0=\nabla \cdot F_{k,0}=0.
$$
Then equation \eqref{newIDEAL} has a solution $u,F_k$ with respective initial data $u_0$, $F_{k,0}$ belonging to the space
\begin{equation}\label{SPACE}
C([0,T], H^s(\R^3))\cap C^1([0,T], H^{s-1}(\R^3)),
\end{equation}
provided
\begin{equation}\label{BKMcond}
\int^T_0 |(\nabla\times u)(t)|_{L^\I(\R^3)} +\sum^3_{k=1}|(\nabla \times F_k)(t)|_{L^\I(\R^3)}dt<\infty.
\end{equation}
In particular, solutions with smooth initial data remain smooth on $(0,T)$ provided \eqref{BKMcond} holds.
\end{thm}

\begin{proof}
The short time existence of solutions $u, F_k$ with respective
initial data $u_0$, $F_{k,0}$ belonging to the space \eqref{SPACE}
for some $T>0$ follows from routine arguments; it is
straightforward to adapt the arguments in chapter 3, section 2 of \cite{Majda}, for
instance. Moreover, if  $u,F_k$ do not belong to the space
\eqref{SPACE} then it must be that
$$
\limsup_{t\rightarrow T^-}\left(|u(t)|_{H^s(\R^3)}+\sum^3_{k=1}|F_k(t)|_{H^s(\R^3)}\right)=+\I.
$$
If this is not the case, there is $C_0>0$ such that $|u(t)|_{H^s(\R^3)}+\sum^3_{k=1}|F_k(t)|_{H^s(\R^3)}\le C_0$ for  $t<T$. Moreover, letting $t_1<T$ we can solve \eqref{newIDEAL} on the interval $(t_1,t_1 + T_0)$
for some $T_0>0$ only depending on $C_0$ and, in particular, independent of $t_1.$  By choosing $t_1\in (T-T_0,T)$, we have extended the original solution past the time $T$. However, in this scenario $u,F_k$ already belonged to the space \eqref{SPACE} by the local
existence of solutions.

\par Therefore, it suffices to show that if \eqref{BKMcond} holds then
\begin{equation}\label{FinalEst}
|u(t)|^2_{H^s}+\sum^3_{k=1}|F_k(t)|^2_{H^s}\le C, \quad 0\le t\le T
\end{equation}
for some constant $C$ depending only on $T$, the $H^s(\R^3)$ norms of the initial data, and
$$
M:=\int^T_0 |(\nabla\times u)(t)|_{L^\I(\R^3)} +\sum^3_{k=1}|(\nabla \times F_k)(t)|_{L^\I(\R^3)}dt<\I.
$$
As in \cite{BKM} and \cite{Fish},  the proof is accomplished in three steps: obtaining energy estimates on solutions, obtaining $L^2$ estimates for the curl of solutions, and applying the crucial estimate \eqref{BIOT}.  The proof is
virtually the same for any integer $s\ge 3$, so for simplicity we assume that $s=3$. For ease of notation, we will also denote $L^\I(\R^3), L^2(\R^3)$, and $H^3(\R^3) $ by $L^\I, L^2$, and $H^3$. We shall also write $\sum_k$ for $\sum^3_{k=1}$.
Finally, all integrals will be taken over the entire space $\R^3.$

\par {\bf Step 1} (Energy estimates).  We first show that there is a universal constant $C$ such that
\begin{equation}\label{ENERGY}
|u(t)|^2_{H^3}+\sum_k|F_k(t)|^2_{H^3}\le \left(|u_0|^2_{H^3}+\sum_k|F_{k,0}|^2_{H^3}\right)\exp\left(C \int^t_0|\nabla u|_{L^\infty}+\sum_k|\nabla F_k|_{L^\infty}ds\right)
\end{equation}
for $0\le t\le T$.  Let $\alpha$ be a multi-index with $|\alpha|\le 3$ and set
$$
v:= \partial_x^\alpha u\quad \text{and}\quad G_k:= \partial_x^\alpha F_k
$$
for $k=1,2,3.$ Note that

$$
\begin{cases}
\hspace{.23in}\partial_t v + (u\cdot \nabla) v = -\nabla(\partial_x^\alpha p) + \sum^3_{k=1}(F_k\cdot\nabla)G_k - Q\\
\partial_t G_k + (u\cdot \nabla)G_k = (F_k\cdot \nabla)v - R_k
\end{cases}
$$
where
$$
\begin{cases}
Q=\partial_x^\alpha((u\cdot \nabla)u) - (u\cdot \nabla)\partial_x^\alpha u-\sum_k\left[\partial_x^\alpha((F_k\cdot \nabla)F_k) - (F_k\cdot \nabla)\partial_x^\alpha F_k\right]\\
R_k=\partial_x^\alpha((u\cdot \nabla)F_k) - (u\cdot \nabla)\partial_x^\alpha F_k - [\partial_x^\alpha((F_k\cdot \nabla)u) - (F_k\cdot \nabla)\partial_x^\alpha u], \quad k=1,2,3
\end{cases}.
$$
As $\nabla\cdot u=\nabla\cdot F_k=0$, we have
\begin{align*}
\frac{d}{dt}\int \frac{|v|^2}{2}&=\int v\cdot \partial_t v \\
&= \int v\cdot \left( \sum_k(F_k\cdot\nabla)G_k - Q\right)\\
&= -\int  \sum_kG_k\cdot(F_k\cdot\nabla)v - \int v\cdot Q\\
&\le -\int  \sum_kG_k\cdot(F_k\cdot\nabla)v + |v|_{L^2}|Q|_{L^2}.
\end{align*}
Likewise,
\begin{align*}
\frac{d}{dt}\int \sum_k\frac{|G_k|^2}{2}&=\int \sum_kG_k\cdot \partial_t G_k \\
&=\int \sum_k-G_k \cdot R_k+ G_k\cdot(F_k\cdot\nabla)v \\
&\le \sum_k|G_k|_{L^2} |R_k|_{L^2}+ \int G_k\cdot(F_k\cdot\nabla)v.
\end{align*}
Consequently,

$$
\frac{d}{dt}\int \sum_k\frac{|v|^2}{2}+\frac{|G_k|^2}{2}\le |v|_{L^2}|Q|_{L^2}+\sum_k|G_k|_{L^2} |R_k|_{L^2}.
$$
Recall the identity,
$$
|\partial^\alpha_x(fg) - f\partial^\alpha_xg|_{L^2}\le C\{|f|_{H^3}|g|_{L^\infty} + |\nabla f|_{L^\infty}|g|_{H^{2}}\}
$$
for smooth $f,g$ and $|\alpha|\le 3$ \cite{BKM}; here $C$ is an absolute constant. This identity implies

$$
\begin{cases}
|Q|_{L^2}\le C\left\{|u|_{H^3}|\nabla u|_{L^\I}+\sum_k|F_k|_{H^3}|\nabla F_k|_{L^\I}\right\}\\
|R_k|_{L^2}\le C\left\{|u|_{H^3}|\nabla F_k|_{L^\I}+|F_k|_{H^3}|\nabla u|_{L^\I}\right\}, \quad k=1,2,3
\end{cases}.
$$
Hence,
\begin{align*}
\frac{d}{dt}\left(\frac{|v|^2}{2}+\sum_k\frac{|G_k|^2}{2}\right)& \le |v|_{L^2}\left(|u|_{H^3}|\nabla u|_{L^\I} +\sum_k|F_k|_{H^3}|\nabla F_k|_{L^\I}\right) \\
& + \sum_k|G_k|_{L^2}\left(|u|_{H^3}|\nabla F_k|_{L^\I} +|F_k|_{H^3}|\nabla u|_{L^\I}\right).
\end{align*}
Summing over all multi indices $|\alpha|\le 3$ gives
$$
\frac{d}{dt}\left(|u|_{H^s}^2+\sum_k|F_k|_{H^s}^2\right)\le C
\left(|\nabla u|_{L^\I} +\sum_k|\nabla F_k|_{L^\I}
\right)\left(|u|_{H^s}^2+\sum_k|F_k|_{H^s}^2\right)
$$
for some universal constant $C$. Inequality \eqref{ENERGY} now follows from Gronwall's inequality.

\par {\bf Step 2} ($L^2$ estimates).  Set
$$
w:=\nabla \times u\quad \text{and}\quad r_k:=\nabla \times F_k, \quad k=1,2,3.
$$
We claim
\begin{equation}\label{curlL2}
|w(t)|^2_{L^2}+\sum_k|r_k(t)|^2_{L^2}\le \left(|w_0|^2_{L^2}+\sum_k|r_{k,0}|^2_{L^2}\right)\exp\left(C\int^t_0 |w|_{L^\I}+\sum_k|r_k|_{L^\I}ds\right)
\end{equation}
for $0\le t\le T$ and some universal constant $C$.  With hypothesis \eqref{BKMcond}, the above inequality implies an $L^2$ bound on $w(t)$ and $r_k(t)$ that is independent of $t\in [0, T]. $   On our way to proving this bound, we recall a simple fact:
$u,v \in H^2(\R^3;\R^3)$, then
$$
\nabla\times[(u\cdot \nabla)v]=a(\nabla u,\nabla v) + (u\cdot \nabla)(\nabla \times v)
$$
where $a: \R^{3\times 3}\times \R^{3\times 3}\rightarrow \R^3$ is given by
$$
a(X,Y):=\sum^3_{i=1}(X^te_i)\times (Ye_i), \quad X,Y\in \R^{3\times 3}.
$$
Taking the curl of the first two equations in \eqref{newIDEAL} gives
$$
\begin{cases}
\hspace{.05in}\partial_t w+ (u\cdot \nabla) w =   \sum_k(F_k\cdot\nabla)r_k - S\\
\partial_t r_k + (u\cdot \nabla)r_k = (F_k\cdot \nabla)w - T_k
\end{cases},
$$
where
$$
\begin{cases}
S=a(\nabla u, \nabla u) - \sum_k a(\nabla F_k,\nabla F_k)\\
T_k=a(\nabla u, \nabla F_k)-a(\nabla F_k, \nabla u), \quad k=1,2,3.
\end{cases}
$$
We compute
\begin{align*}
\frac{d}{dt}\int \frac{|w|^2}{2}&=\int w\cdot \partial_t w \\
&= \int w\cdot \left( \sum_k(F_k\cdot\nabla)r_k - S\right)\\
&= -\int  \sum_kr_k\cdot(F_k\cdot\nabla)w - \int v\cdot S.
\end{align*}
Likewise,
\begin{align*}
\frac{d}{dt}\int \sum_k\frac{|r_k|^2}{2}&=\int \sum_kr_k\cdot \partial_t r_k \\
&=\int \sum_k-r_k \cdot T_k+ r_k\cdot(F_k\cdot\nabla)w.
\end{align*}
Consequently,
\begin{align*}
\frac{d}{dt}\int \frac{|w|^2}{2}+\sum_k\frac{|r_k|^2}{2}&= -\int\left\{w\cdot S + \sum_k r_k\cdot T_k\right\}\\
&\le |\nabla u|_{L^\I}|w|_{L^2}|\nabla u|_{L^2} + \sum_k|\nabla F_k|_{L^\I}|w|_{L^2}|\nabla F_k|_{L^2}\\
& +  \sum_k\left(|\nabla u|_{L^\I}|r_k|_{L^2}|\nabla F_k|_{L^2} + |\nabla F_k|_{L^\I}|r_k|_{L^2}|\nabla u|_{L^2}\right)\\
&\le C\left( |\nabla u|_{L^\I}+\sum_k |\nabla F_k|_{L^\I} \right)\left(|w|^2_{L^2}+ \sum_k |r_k|^2_{L^2}\right)
\end{align*}
as
$$
|\nabla v|_{L^2}\le C|\nabla \times v|_{L^2}
$$
for $v\in H^1$, $\nabla \cdot v=0$. Inequality \eqref{curlL2} now follows from Gronwall's inequality.

\par {\bf Step 3} (Applying Kato's inequality).  Now we are in a position to complete the proof.  Assume that \eqref{BKMcond} holds so that \eqref{curlL2} implies $w,r_k\in L^\I((0,T); L^2).$
By Kato's inequality \eqref{BIOT}, we have
$$
|\nabla u|_{L^\I} + \sum_k|\nabla F_k|_{L^\I}\le C\left\{ 1 +
\left(|w|_{L^\I}+\sum_k|r_k|_{L^\I}\right)\left(\ln(|u|_{H_3}+e)+\ln\left(\sum_k|F_k|_{H_3}+e\right)\right)
\right\}.
$$
By the energy estimate \eqref{ENERGY}, we also have
\begin{align*}
|u|_{H^3}+ \sum_k| F_k|_{H^3}&\le \left(|u_0|_{H^3}+ \sum_k|
F_{k,0}|_{H^3}\right)\exp\Biggl(C\int^t_0
\Biggl\{1+\left(|w|_{L^\I}+\sum_k|r_k|_{L^\I}\right)\\&\quad\times\left(\ln(|u|_{H_3}+e)+\ln\left(\sum_k|F_k|_{H_3}+e\right)\right)\Biggr\}ds\Biggr).
\end{align*}
Now setting
$$
y(t):=\ln\left(|u(t)|_{H^3}+e\right)+\ln\left(\sum_k| F_k(t)|_{H^3}+e\right), \quad 0\le t\le T
$$
gives
\begin{equation*}
y(t)\le B + C\int^t_0 \left\{1 +
\left(|w(s)|_{L^\I}+\sum_k|r_k(s)|_{L^\I}\right) y(s)\right\}ds.
\end{equation*}
Here $C$ is a universal constant and $B$ depends only on the $H^3$ norms of the initial data. By Gronwall's inquality, we conclude that $y(t)$ is bounded by a constant depending on $M$, $T$ and the $H^3$ norms of the inital data.  This bound establishes inequality \eqref{FinalEst} and completes
the proof of the theorem.
\end{proof}
Finally, we remark that results closely related to Theorem \ref{mainthm} have been obtained for two dimensional flow \cite{MAS} and for three dimensional flow in the so-called ``creeping flow" regime \cite{KUF}.

\appendix

\newpage


\end{document}